\documentclass[11pt,reqno]{amsart}
\usepackage {amssymb,color}
\usepackage{verbatim}
\input xy
\xyoption{all}
\input epsf

\newcommand{\sgn}{\mathrm{sgn}}
\newcommand{\triv}{\mathsf{1}}

\newcommand{\GL}{\mathrm{GL}}
\newcommand{\U}{\mathrm{U}}
\newcommand{\Hom}{\mathrm{Hom}}

\newcommand{\ol}{\overline}

\newcommand{\field}{\mathbb}
\newcommand{\liealgebra}{\mathfrak}

\newcommand{\C}{{\field C}}
\newcommand{\R}{{\field R}}

\newcommand{\lev}{{\mathrm{lev}}}

\newcommand{\p}{\liealgebra p}
\renewcommand{\b}{\liealgebra b}

\newcommand{\lam}{\lambda}
\newcommand{\gam}{\gamma}

\newcommand{\al}{\alpha}

\newcommand{\SO}{\mathrm{SO}}
\newcommand{\SL}{\mathrm{SL}}

\renewcommand{\H}{\mathrm{H}}

\newcommand{\lra}{\longrightarrow}

\newcommand{\bs}{\backslash}

\renewcommand{\-}{\! - \!}

\newcommand{\std}{\textrm{std}}
\newcommand{\irr}{\textrm{irr}}

\newtheorem{prop}{Proposition}[section]

\newtheorem{lemma}[prop]{Lemma}
\newtheorem{theorem}[prop]{Theorem}

\numberwithin{equation}{section}

\newtheorem{corollary}[prop]{Corollary}
\newtheorem{proposition}[prop]{Proposition}

\theoremstyle{definition}

\newtheorem{remark}[prop]{Remark}
\newtheorem{example}[prop]{Example}
\newtheorem{notation}[prop]{Notation}

\newcommand{\frg}{\mathfrak{g}}
\newcommand{\frh}{\mathfrak{h}}

\newcommand{\frk}{\mathfrak{k}}
\newcommand{\frl}{\mathfrak{l}}

\newcommand{\frn}{\mathfrak{n}}

\newcommand{\frp}{\mathfrak{p}}

\newcommand{\frs}{\mathfrak{s}}

\newcommand{\bbC}{\mathbb{C}}

\newcommand{\bbH}{\mathbb{H}}

\newcommand{\bbN}{\mathbb{N}}

\newcommand{\bbQ}{\mathbb{Q}}
\newcommand{\bbR}{\mathbb{R}}

\newcommand{\bbZ}{\mathbb{Z}}

\newcommand{\caC}{\mathcal{C}}
\newcommand{\caD}{\mathcal{D}}

\newcommand{\caF}{\mathcal{F}}
\newcommand{\caG}{\mathcal{G}}
\newcommand{\caH}{\mathcal{H}}
\newcommand{\caI}{\mathcal{I}}

\newcommand{\caM}{\mathcal{M}}

\newcommand{\caP}{\mathcal{P}}

\newcommand{\beq}{\begin{equation}}
\newcommand{\eeq}{\end{equation}}

\def \Ad  {\mathop{\hbox {Ad}}\nolimits}
\def\ad {\mathop{\hbox {ad}}\nolimits}

\def \det {\mathrm{det}}
\def \End {\mathop{\hbox {End}}\nolimits}

\def\Ind {\mathop{\hbox{Ind}}\nolimits}
\def\Res {\mathop{\hbox{Res}}\nolimits}

\def\Id {\mathop{\hbox{Id}}\nolimits}

\newcommand{\sph}{{\mathrm{sph}}}
\newcommand{\HC}{\caH\caC}
\newcommand{\HCS}{\caH\caC^\sph}
\newcommand{\catI}{\caI^\sph}

\renewcommand{\H}{\caH}

\renewcommand{\Re}{\mathrm{Re}}
\renewcommand{\p}{\mathrm{per}}
\renewcommand{\c}{\mathrm{con}}

\newcommand{\St}{\mathrm{St}}

\newcommand{\pos}[1]{{\mathrm{pos}({#1})}}

\newcommand{\PR}{\caP^\R}
\newcommand{\gPR}{\caP^{\R,g}}
\newcommand{\PH}{\caP^\bbH}
\newcommand{\gPH}{\caP^{\bbH,g}}
\newcommand{\pPR}{{}'\PR}
\newcommand{\pPH}{{}'\PH}
\newcommand{\X}{X}
\newcommand{\frgone}{\frg_{-1}}

\begin{document}
\title{Duality between $\GL(n,\R)$ and the degenerate affine Hecke algebra
for $\frg\frl(n)$}
\author{DAN CIUBOTARU and PETER E.~TRAPA}
\date{\today}
\address{Department of Mathematics, University of Utah, Salt Lake City, UT 84112-0090}
\email{ciubo@math.utah.edu}
\email{ptrapa@math.utah.edu}
\maketitle

{
\begin{abstract}
We define an exact functor $F_{n,k}$ from the category of Harish-Chandra
modules for $\GL(n,\R)$ to the category
of finite-dimensional representations for the degenerate
affine Hecke algebra for $\frg\frl(k)$.  Under certain natural
hypotheses, we
prove that the functor
maps standard modules to standard modules (or zero) and irreducibles
to irreducibles (or zero).
\end{abstract}
}

\section{introduction}
\label{s:intro}

In this paper, we define an exact functor $F_{n,k}$ from the category $\HC_n$ of Harish-Chandra
modules for $G_\R = \GL(n,\R)$ to the category $\H_k$
of finite-dimensional representations for the degenerate
affine Hecke algebra $\bbH_k$ for $\frg\frl(k)$.  When we take $k=n$ and
{restrict to an appropriate subcategory}, we
prove that the functor
maps standard modules to standard modules (or zero) and irreducibles
to irreducibles (or zero).   We deduce the latter statement from the former using
a geometric relationship between
unramified Langlands parameters for $\GL(n,\bbQ_p)$ 
and Langlands parameters for $\GL(n,\R)$ (or rather the Adams-Barbasch-Vogan
version of them).
Our functor may be viewed as a real version of
the one defined by Arakawa and Suzuki \cite{AS}, and the geometric statement may be 
viewed as a real version of \cite{zel} due (independently)
to Lusztig and Zelevinsky.

The functor is very simple to define.  Let $K_\R = O(n)$, a maximal compact
subgroup of $G_\R$, write $\frg = \frg\frl(n,\C)$ for the complexified Lie algebra,
and let $V$ denote the standard representation of $G_\R$.  Let $\sgn$ denote
the determinant representation of $K_\R$.
Given a Harish-Chandra module $X$ for $G_\R$, we define
\[
F_{n,k}(X) = \Hom_{K_\R}(\triv , (X \otimes \sgn) \otimes V^{\otimes k}).
\]
(The twist of $X$ by $\sgn$ is a convenient normalization and is not conceptually important.)
It is known that $\bbH_k$ acts on $Y \otimes V^{\otimes k}$
for any $\U(\frg)$-module $Y$ (e.g.~\cite[2.2]{AS}); see Section \ref{s:fnk} below.
In our setting, it is easy to see that this
action commutes with $K_\R$, and thus $F_{n,k}(X)$ becomes a
module for $\bbH_k$.  Obviously $F_{n,k}$ is exact and covariant.
Related functors appear in the work of Etingof, Freund, and Ma (\cite{EFM}, \cite{ma}),
and in that of {Oda \cite{oda}}.

We note that
(even if we take $k=n$) 
the functor $F_{n,k}$ does not behave well on the category of all Harish-Chandra
modules for $G_\R$.   This is not surprising.  After all, using the Borel-Casselman
equivalence \cite{borel} and the reduction of Lusztig \cite{lu:graded}, we may interpret $\H_n$
as the category $\catI_n$ of Iwahori-spherical representations of the $p$-adic
group $\GL(n,\bbQ_p)$.  The objects in this latter category are exactly
the subquotients of spherical principal series.  So it is natural to expect that
$F_{n,n}$ should only be well-behaved on some real analog of $\catI_n$, and this is indeed
the case.  We make this more precise in Section \ref{ssec:lev} where we introduce
a notion of level for $\HC_n$, and define a category $\HC_{n,\geq k}$ consisting
of modules of level at least $k$. {(For instance, every subquotient of a spherical
principal series for $G_\R$ is an object in $\HC_{n,\geq n}$; see Example \ref{ex:sph}.)  }
We prove that $F_{n,k}$ maps standard modules in $\HC_{n,\geq k}$
to standard modules in $\H_k$ (or zero).  When $k = n$, we further prove that $F_{n,n}$ maps 
irreducibles to irreducible (or zero).

\begin{theorem}
\label{t:intro}
Suppose $X$ is an irreducible Harish-Chandra module for $\GL(n,\R)$
whose level is at least $n$.  
Then $F_{n,n}(X)$
is irreducible or zero.  {Moreover, $F_{n,n}$ implements a bijection
between irreducible Harish Chandra modules of level exactly $n$
and irreducible $\bbH_n$-modules.} 
\end{theorem}

Here is a sketch of the proof of Theorem \ref{t:intro}.
Since the theorem is about a nice relationship between $\GL(n,\R)$
and $\GL(n,\bbQ_p)$ (in the guise of $\H_n$), 
the place to begin looking for its origins is on the level of Langlands parameters.
This is the setting of Section \ref{s:geom}.  The main
result there is Theorem \ref{t:geom},  
a suitably equivariant compactification of the space of unramified Langlands parameters
for $\GL(n,\bbQ_p)$ by spaces of ABV parameters.
As a consequence (Corollary \ref{c:geom}), we obtain
that various coefficients in the expression of irreducible
modules in terms of standard ones coincide in both $\HC_n$
and $\H_n$.
The conclusion is that if one could find an exact  functor which
matches the ``right'' standard modules in 
both cases, it would automatically match irreducibles too.  
In Section \ref{s:std}, we make the relevant computation of $F_{n,k}$
applied to
standard modules (Theorem \ref{t:stdcomp}), 
and in Section \ref{s:irr} we check that
the matching is the right one from the viewpoint of
the geometry of Section \ref{s:geom}.  The statement
about irreducibles mapping to irreducibles (or zero)
follows immediately (Corollary \ref{c:irrcomp}).

The functor $F_{n,k}$ has a number of other good properties which
we shall pursue in detail elsewhere. 
For instance, 
$F_{n,k}$ takes certain special unipotent derived functor modules to 
interesting unitary representations defined by Tadi\'c; see Example \ref{ex:basic}(2).  
More generally,  
it matches appropriately
defined Jantzen filtrations in the real and $p$-adic cases (a real
version of the results of Suzuki \cite{suzuki}).   
We will use this fact to study unitary representations, 
ultimately giving a functorial explanation of the
coincidence of the spherical unitary duals of $\GL(n,\R)$
and $\GL(n,\bbQ_p)$
(\cite{vogan:gln}, \cite{ta}, \cite{ta2}, \cite{ba}).

\section{Geometric relationship between the 
langlands classification\\
 for $\GL(n,\R)$ and $\bbH_n$}
\label{s:geom}

\subsection{The Langlands classification for $\GL(n,\R)$.}
\label{ssec:glclass}
We begin with the classification of irreducible objects in $\HC_n$
which does not involve the dual group.  In order to do so, we must
recall the relative discrete series of $\GL(1,\R)$ and $\GL(2,\R)$,
and accordingly we must discuss representations of the maximal
compact subgroups $O(1)$ and $O(2)$.  
As in the introduction, we continue to write $\sgn$ for the 
determinant representation of $O(n)$.  
Apart from $\triv$ and $\sgn$, the remaining irreducible representations of $O(2)$
are two-dimensional and parametrized by integers $n \geq 1$.
We let $V(n)$
denote the irreducible representation of $O(2)$ with $\SO(2)$
weights $\pm n$.  It is also convenient to let $V(0)$
denote the reducible representation $\triv \oplus \sgn$.  Then
we always have $V(k) \otimes V(l) \simeq V(|k-l|) \oplus V(k+l)$ for instance.

For $x \in \GL(1,\R) \simeq \bbR^\times$,
write $\sgn(x)$ for the sign of $x$.  Then any irreducible
representation of $\bbR^\times$ is a relative 
discrete series and is of the form
\begin{equation}
\label{e:gl1ds}
\delta(\varepsilon, \nu) := \varepsilon \otimes |\cdot|^\nu.
\end{equation}
for $\varepsilon \in \{\triv, \sgn\}$ and $\nu \in \bbC$.
Meanwhile, any relative discrete series for $\GL(2,\R)$ is of the form
\begin{equation}
\label{e:gl2ds}
\delta(l,\nu):=D_l \otimes |\det(\cdot)|^\nu,
\end{equation}
where $l \in \bbZ^{\geq 2}$, $\nu \in \bbC$, $\det$ is the determinant character, and $D_n$
is a discrete series representation of $\SL^\pm(2,\R)$
(the group of two-by-two real matrices with determinant
$\pm1$) with lowest $O(2)$-type $V(l)$.  In more detail,
$D_l$ is characterized by requiring its restriction to $O(2)$
decompose as the sum $V(l+2k)$ over $k\in \bbN$.

We now introduce a key parameter set $\pPR_n$.  Its elements consists
of pairs $(P_\R,\delta)$.  Here $P_\R$ is a block upper triangular
subgroup of $G_\R$ whose Levi factor is an (ordered) product 
\[
L_\R = \GL(n_1,\R) \times \cdots \times \GL(n_r,\R)
\]
with $n_i \in {1,2}$, and $\delta = \delta_1 \boxtimes \cdots \boxtimes \delta_r$
is a relative discrete series of $L_\R$.  Thus each $\delta_i$ is of the form
$\delta(\varepsilon_i,\nu_i)$ (as in \eqref{e:gl1ds}) if $n_i =1$, and
otherwise of the form $\delta_i=\delta(l_i,\nu_i)$ (as in \eqref{e:gl2ds}).
We impose the further condition that
\begin{equation}
\label{e:dom}
n_1^{\-1}\Re(\nu_1) \geq n_2^{-1}\Re(\nu_2) \geq \cdots \geq
n_r^{\-1}\Re(\nu_r).
\end{equation}
To each such pair $\gamma'=(P_\R,\delta)$ in $\pPR_n$, we may form the 
parabolically induced standard module
\begin{equation}
\label{e:realstd}
\std(\gamma') := \Ind_{P_\R}^{G_\R}(\delta);
\end{equation}
here it is understood that $\delta$ has been extended trivially to the nilradical of $P_\R$.
It is further understood that the induction is normalized as in \cite[Chapter VII]{knapp}.
The condition \eqref{e:dom} guarantees that $\std(\gamma')$ has a unique irreducible
quotient, which we denote $\irr(\gamma')$.  Alternatively, $\irr(\gamma')$ is characterized
as the constituent of $\std(\gamma')$ containing its (unique) lowest $K$-type.

The assignment $\gamma' \mapsto \irr(\gamma')$ is not quite injective.  To remedy this
we let $\PR_n$ denote the set of equivalence classes in $\pPR_n$ for the relation $(P_\R,\delta) \sim
(P'_\R,\delta')$ if the two differ by the obvious kind of rearrangement of factors.
Then $\std(\gamma')$ (and thus $\irr(\gamma')$) depend only on the equivalence class of $\gamma'$.
It thus makes sense to write $\std(\gamma)$ and $\irr(\gamma)$ for $\gamma \in \PR_n$.
When we want to emphasize that we are in the real case, we may write $\std_\R(\gamma)$
and $\irr_\R(\gamma)$ instead.

Here is the classical Langlands classification in this setting (\cite{la}, cf.~\cite[Section 2]{v:unit}).

\begin{theorem}\label{t:glclass}
With notation as above, the map
\begin{align}
\label{e:glclass} 
\PR_n &\longrightarrow \text{irreducible objects in $\HC_n$} \nonumber\\
\gamma &\longrightarrow \irr_\R(\gamma)
\end{align}
is bijective.
\end{theorem}

In the Grothendieck group of $\HC_n$ (where we denote the image
of an object $M$ by $[M]$), we may consider expressions of the form
\begin{equation}
\label{e:gldecomp}
{
[\irr_\R(\eta)] = \sum_{\gamma \in \PR_n} M_\R(\gamma, \eta) [\std_\R(\gamma)];
}
\end{equation}
here $M_\R(\gamma,\eta) \in \bbZ$. Each such expression is
finite.  More precisely,
fix a Cartan subalgebra $\frh$ of $\frg$, and fix $\lam \in \frh^*$.  According
to the Harish-Chandra isomorphism, $\lam$ defines an infinitesimal character
for $\frg$.  Let $\HC_n(\lam)$ denote the full subcategory of modules with this
infinitesimal character; there are only finitely many irreducible objects
in $\HC_n(\lam)$.   Using the classification of Theorem \ref{t:glclass}, let
$\PR_n(\lam)$ denote the parameters $\gamma$ such that $\irr_\R(\gam)$
is an object of $\HC_n(\lam)$.  Then if $\gamma \in \PR_n(\lam)$ and
$\eta \notin \PR_n(\lam)$, we have $M_\R(\gamma,\eta) = 0.$

In the next section we give 
a geometric interpretation of the numbers $M_\R(\gamma,\eta)$.

\subsection{Geometry of the Langlands classification for $\GL(n,\R)$.}
\label{ssec:glgeom}
One natural approach to computing the numbers $M_\R(\gamma,\eta)$ 
of \eqref{e:gldecomp} involves the Beilinson-Bernstein localization
functor from $\HC_n(\lam)$ to
$O(n,\C)$-equivariant $\lam$-twisted $\caD$-modules on the flag variety of $\frg$
\cite{v:ic3}.  
But there is no analogue of this kind of localization in the $p$-adic case.  
From the viewpoint of the local Langlands
conjecture, it is instead more natural to work with the geometry of the
reformulated space of Langlands parameters due to \cite{abv}.

Though not necessary, we find it convenient to
work with one infinitesimal character at a time.  
As above fix a Cartan subalgebra $\frh$ of $\frg$, and fix $\lam \in \frh^*$.
At this point we have two options.  We could identify
$\frh^*$ with $\frh$ (using the trace form, for instance), and view $\lam$
as an element of $\frh$.  Alternatively, we could canonically identify
$\frh^*$ with a Cartan subalgebra of the Lie algebra $\frg^\vee$ of the complex Langlands
dual group $G^\vee$.  Of course $\frg^\vee \simeq \frg\frl(n,\C)$, and so both options
are equivalent.   To keep notation to a minimum, we choose the first route, and henceforth
consider $\lam$ as a semisimple element of $\frg$.  But it is important to keep in mind
that the geometry we introduce below is naturally defined ``on the dual side'' (a fact which
is particularly important when considering generalizations outside of Type A).

Consider $\ad(\lam)$.  Let $\frg(\lam)$ denote the sum of its integral eigenspaces, 
let $\frn(\lam)$ denote the sum of its strictly positive integral eigenspaces, and let
$\frl(\lam)$ denote its zero eigenspace.  Set $\frp(\lam) = \frl(\lam) \oplus \frn(\lam)$.
Set $y(\lam) = \exp(\pi i\lam)$ and $e(\lam) = y(\lam)^2 = \exp(2\pi i \lam)$.   
Write $G(\lam)$ for the centralizer in $G$ of $e(\lam)$; clearly its Lie algebra
is $\frg(\lam)$.  Write $L(\lam)$ for the centralizer in $G$ of $\lam$; its Lie algebra
is $\frl(\lam)$.  Let $P(\lam)$ denote the analytic subgroup of $G$ with
Lie algebra $\frp(\lam)$.  Finally let $K(\lam)$ denote the centralizer in $G(\lam)$
of $y(\lam)$; it's Lie algebra $\frk(\lam)$ is the sum of the {\em even} integral eigenspaces
of $\ad(\lam)$.  Since $y(\lam)$ squares to $e(\lam)$, $K(\lam)$ is a symmetric
subgroup of $G(\lam)$.  {
For instance if $\lam=\rho$ corresponds to the trivial infinitesimal
character, then $G = G(\lam)$ and 
$K(\lam) \simeq \GL(\lceil{\frac n2}\rceil,\C) 
\times \GL(\lfloor{\frac n2}\rfloor,\C)$.

In practice, only the symmetric subgroup $K(\lam)$ will arise for us.  But to formulate
Theorem \ref{t:glgeomclass} we need others (in order to account for other ``blocks''
of representations for $G_\R$).  Let $\{y_0, \dots, y_r\}$ denote
representatives of $G$ conjugacy classes of semisimple elements which
square to $e(\lam)$.  Arrange the ordering so that $y_0 = y(\lam)$ above,
let $K_i(\lam)$ denote the centralizer in $G$ of $y_i$.  In the example of $\lam = \rho$
mentioned above, the collection $\{K_i(\lam)\}$ equals $\{\GL(p,\C) \times \GL(q,\C)
\; | \; p+q=n\}$.}

We need to introduce notation for intersection homology, and it is convenient
to do this in greater generality.
So suppose $H$ is a complex algebraic group acting with
finitely many orbits on a complex algebraic variety $X$.  Let $\phi$ be an
irreducible $H$-equivariant local system on $X$.  Then $\phi$ naturally parametrizes
both an irreducible $H$-equivariant constructible sheaf $\c(\phi)$ on $X$ and an irreducible
$H$-equivariant perverse sheaf $\p(\phi)$ on $X$ (\cite[1.4.1, 4.3.1]{bbd}).  {By taking
Euler characteristics, we} may
identify the integral Grothendieck group of the categories of $H$-equivariant perverse sheaves on $X$ and 
$H$-equivariant constructible
sheaves on $X$.  In this Grothendieck group, we may write
\begin{equation}
\label{e:gengeomdecomp}
[\c(\phi)] = (-1)^{l(\phi)}\sum_\psi M^g(\psi, \phi) [\p(\psi)];
\end{equation}
here the sum is over all $H$-equivariant local systems $\psi$, and $l(\phi)$ denotes the dimension
of the support of $\phi$.  (The superscript ``$g$'' is meant to stand for ``geometric''.)

{We return to our specific setting and let $\gPR_n(\lam)$ denote the disjoint union over
$i\in \{0,\dots, r\}$ of the set of irreducible $K_i(\lam)$ equivariant
local systems on $\X(\lam) = G(\lam)/P(\lam)$.    
In fact, the centralizer in each $K_i(\lam)$ 
of any point in $\X(\lam)$ always turns out to be connected,
and thus each such local system
is trivial.  As a matter of notation, if $Q$ is an orbit of $K_i(\lam)$
on $\X(\lam)$, then we will also write $Q$ for the corresponding trivial local system; in particular, we will
write things like $Q\in \gPR_n(\lam)$ and implicitly identify $\gPR_n(\lam)$ with $\bigcup_i K_i(\lam) \bs \X(\lam)$.  
Specializing \eqref{e:gengeomdecomp} to this setting (taking $H$ to be any of the groups $K_i(\lam)$ and $X= \X(\lam)$),} we thus obtain
integers $M_\R^g(Q,Q')$ for $K_i(\lam)$ orbits $Q$ and $Q'$ on $\X(\lam)$ defined via
\begin{equation}
\label{e:glgeomdecomp}
[\c(Q)] = (-1)^{\dim(Q)}\sum_{Q' \in \gPR_n(\lam)} M_\R^g(Q', Q) [\p(Q')];
\end{equation}
The following is the geometric formulation of the local Langlands correspondence is our setting
(\cite[Corollary 1.26(c), Corollary 15.13(b)]{abv}).

\begin{theorem}
\label{t:glgeomclass}  
Fix an infinitesimal character $\lam$ as above, recall the parameter space $\PR_n(\lam)$
of Section \ref{ssec:glclass}, and recall the integers $M_\R(\gamma,\eta)$ and $M_\R^g(Q,Q')$ of
\eqref{e:gldecomp} and \eqref{e:glgeomdecomp}.  Then there is a bijection
\[
d_\R \; : \; \PR_n(\lam) \longrightarrow \gPR_n(\lam)
\]
such that
\[
M_\R(\gamma,\eta) = \pm
M_\R^g(d_\R(\eta),d_\R(\gamma));
\]
here the sign is that of the parity of the difference in dimension of
(the support of)
$d_\R(\gamma)$ and $d_\R(\eta)$.
\end{theorem}

\subsection{The Langlands classification for $\bbH_n$.}
\label{ssec:hclass}
The graded Hecke algebra $\bbH_n$ is
an associative algebra with unit defined as follows. 
Fix a Cartan subalgebra $\frh$ in $\frg$ and a system of
simple roots $\Pi(\frg,\frh)$ of $\frh$ in $\frg$.
The $\bbH_n$ contains as
subalgebras the group algebra $\C[W_n]$ of the Weyl group $W_n \simeq S_n$
and the symmetric
algebra $S(\frh)$, subject to the commutation relations
\begin{align}
s_\alpha\cdot \epsilon-s_\alpha(\epsilon)\cdot
s_\alpha=\langle\alpha,\epsilon\rangle,\text{ for all simple roots
}\alpha\in\Pi(\frg,\frh), \text{ and }\epsilon\in \frh.
\end{align}
We write $\H_n$ for the category of finite-dimensional
$\bbH_n$ modules.

We recall the classification of irreducible $\bbH_n$ modules.
To begin, we recall the one-dimensional 
Steinberg module $\St$ on which $\C[W_n]$ acts by the
sign representation and $S(\frh)$ acts
by the weight $-\rho$ corresponding to the choice $\Pi(\frg,\frh)$.
For $\nu \in \bbC$, let $\bbC_\nu$ denote the one dimensional
representation of $\bbH_n$ where $\C[W_n]$ acts trivially and
$S(\frh)$ acts by the weight of the center of $\frg$ corresponding
to $\nu$. 
Any relative discrete series representation of $\bbH_n$
is of the form $\St \otimes \bbC_\nu$.

We next introduce a parameter set $\pPH_n$ analogous to
$\pPR_n$ in Section \ref{ssec:glclass}.  As we implicitly  did there, 
we fix $\frh$ to be the diagonal Cartan subalgebra
and let $\Pi(\frg,\frh)$ correspond to the upper triangular Borel
subalgebra.  We let $\pPH_n$ consist of pairs
$(\bbH_P, \delta)$.  Here $\bbH_P$ is a subalgebra of $\bbH$
corresponding to a block upper triangular subalgebra of $\frg$
whose blocks we write as the ordered product
\[
\frl = \frg\frl(n_1) \oplus \cdots \oplus \frg\frl(n_r),
\]
and $\delta = \delta_1 \boxtimes \cdots \boxtimes \delta_r$
a relative discrete series of $\frl$.  Thus each $\delta_i$
is of the form $\St \otimes \bbC_{\nu_i}$.  We further require
that
\begin{equation}
\label{e:hdom}
\Re(\nu_1) \geq \cdots \geq \Re(\nu_r).
\end{equation}
To each such pair $\gamma' \in \pPH_n$, we form the induced module
\[
\std(\gamma') = \bbH_n \otimes_{\bbH_P} \delta.
\]
Because of \eqref{e:hdom}, $\std(\gamma')$ has a unique irreducible quotient
$\irr(\gamma')$.  

Now let $\PH_n$ denote equivalence classes in $\pPH_n$ for the relation of
rearranging factors (while still preserving \eqref{e:hdom} of course).   Then the
modules $\std(\gamma)$ and $\irr(\gamma)$ are well-defined on equivalence
classes $\gamma \in \PH_n$.  
When we want to emphasize that we are in the Hecke algebra
 case, we may write $\std_\bbH(\gamma)$
and $\irr_\bbH(\gamma)$ instead.

The Langlands classification in this setting is as follows (\cite{bz}, cf.~\cite{evens}).

\begin{theorem}
\label{t:hclass}
With notation as above, the map
\begin{align}
\label{e:hclass} 
\PH_n &\longrightarrow \text{irreducible objects in $\H_n$}\nonumber\\
\gamma &\longrightarrow \irr_\bbH(\gamma)
\end{align}
is bijective.
\end{theorem}

In the Grothendieck group of $\H_n$, we may consider expressions of the form
\begin{equation}
\label{e:hdecomp}
{
[\irr_\bbH(\eta)] = \sum_{\gamma\in \PH_n} M_\bbH(\gamma,\eta) [\std_\bbH(\gamma)];
}
\end{equation}
here $M_\bbH(\gamma,\eta) \in \bbZ$. 
Each such expression is
once again finite.  More precisely, fix $\lam\in \frh^*$.  Since the center
of $\bbH_n$ is $S(\frh)^{W_n}$, $\lam$ defines a central character.
Let $\H_n(\lam)$ denote the full subcategory of modules with this
central character.   Using the classification of Theorem \ref{t:hclass}, let
$\PH_n(\lam)$ denote the parameters $\gamma$ such that $\irr_\bbH(\gam)$
is an object of $\H_n(\lam)$.  Then of course if $\gamma \in \PH_n(\lam)$ and
$\eta \notin \PH_n(\lam)$, we have $M_\bbH(\gamma,\eta) = 0.$

\subsection{Geometry of the Langlands classification for $\bbH_n$.}
\label{ssec:hgeom}
As in the previous section, fix $\lam \in \frh^*$.  With the same
caveats in place as in Section \ref{ssec:glgeom}, we 
identify $\frh^*$ with $\frh$ via the
trace form, and thus view $\lam$ as a semisimple element of $\frg$.

Again as in Section \ref{ssec:glgeom}, we let $L(\lam)$ analytic
subgroup of $G=\GL(n,\C)$ with Lie algebra equal to the zero eigenspace
of $\ad(\lam)$.  We let $\frgone(\lam)$ denote the $-1$ eigenspace.
Then  $L(\lam)$ acts (via the adjoint action) with finitely many orbits
on $\frgone(\lam)$.  These orbits may naturally be identified with orbits
of $G$ on the space of unramified Langlands parameters for
$\GL(n,\bbQ_p)$ (e.g.~\cite[Example 4.9]{vogan:llc}), and so they are to
be considered the $p$-adic analog of the real Langlands parameters
considered in Section \ref{ssec:glgeom}.

Again it transpires that the centralizer in $L(\lam)$
of any point of $\frgone(\lam)$ is connected, and thus we may identity
the set of orbits of $L(\lam)$ on $\frgone(\lam)$ with the set
$\gPH_n(\lam)$ of 
irreducible $L(\lam)$-equivariant local systems on $\frgone(\lam)$.
Specializing \eqref{e:gengeomdecomp} to this setting (taking $H=L(\lam)$
and $X = \frgone(\lam)$),
we obtain
integers $M_\bbH^g(Q,Q')$ defined via
\begin{equation}
\label{e:hgeomdecomp}
[\c(Q)] = (-1)^{\dim(Q)}\sum_{Q' \in \gPH_n(\lam)} M_\bbH^g(Q', Q) [\p(Q')];
\end{equation}
We have the following version of the local Langlands correspondence is 
this setting (\cite[10.5]{lu:cls1}, cf.~\cite[8.6.2]{cg}).

\begin{theorem}
\label{t:hgeomclass}  
Fix a central character $\lam$ as above, recall the parameter space $\PH_n(\lam)$
of Section \ref{ssec:glclass}, and recall the integers $M_\bbH(\gamma,\eta)$ and $M_\bbH^g(Q,Q')$ of
\eqref{e:hdecomp} and \eqref{e:hgeomdecomp}.  Then there is a bijection
\[
d_\bbH \; : \; \PH_n(\lam) \longrightarrow \gPH_n(\lam)
\]
such that
\[
M_\bbH(\gamma,\eta) = \pm M_\bbH^g(d_\bbH(\eta),d_\bbH(\gamma));
\]
here the sign is that of the parity of the difference in dimension of
(the support of)
$d_\bbH(\gamma)$ and $d_\bbH(\eta)$.
\end{theorem}

\subsection{Geometric relationship between Langlands parameters for $\HC_n(\lam)$ and
$\H_n(\lam)$}
\label{ssec:geomrel}
Fix a semisimple element $\lam$ of the diagonal Cartan subalgebra $\frh$.
Consider the injective map
\[
\Phi \; : \; \frgone(\lam) \longrightarrow \X(\lam) = G(\lam)/P(\lam)
\]
defined by 
\begin{equation}
\label{e:Phi}
\Phi(N) = (1 + N)\cdot \frp(\lam).
\end{equation}
Here, as usual, we view $G(\lam)/P(\lam)$ as the variety of conjugates of $\frp(\lam)$.
Then $\Phi$ is obviously equivariant for the action of $L(\lam)$.
Since $L(\lam)$ is a subgroup of $K(\lam)$, if $Q$ is an $L(\lam)$ orbit on $\frgone(\lam)$,
then $\Psi(Q) = K(\lam)\cdot\Phi(Q)$ is a single $K(\lam)$ orbit.  We thus obtain an injection
\begin{equation}
\label{e:Psig}
\Psi^g \; : \; \gPH_n(\lam) \hookrightarrow \gPR_n(\lam).
\end{equation}
{
The image thus does not contain any parameters corresponding to the other symmetric subgroups $K_i(\lam),
i\geq 1$, introduced above.}

The main result of this section is the following.  It may be interpreted as a relation between
the intersection homology of the closures of spaces of real and $p$-adic Langlands parameters
for $\GL(n)$.

\begin{theorem}
\label{t:geom}
In the notation of \eqref{e:glgeomdecomp}  and \eqref{e:hgeomdecomp}, 
\[
M_\R^g(Q,Q') = M_\bbH^g\left(\Psi(Q), \Psi(Q') \right ).
\]
\end{theorem}

\noindent We begin with a simple lemma.

\begin{lemma}  
\label{l:nbar}
Let $\ol{N}(\lam)$ denote the analytic subgroup of $G(\lam)$ with
Lie algebra $\ol{\frn}(\lam)$ consisting of the strictly negative integral eigenvalues of $\ad(\lam)$. 
Then the unipotent group $K(\lam) \cap \ol{N}(\lam)$ acts freely on the image of $\Phi$.  In particular,
given an $L(\lam)$ orbit $Q$ on $\frgone(\lam)$, we have
\[
\dim\left (\Psi(Q) \right ) = \dim\left (Q \right ) + \dim\left (K(\lam) \cap \ol{N}(\lam)Ê\right).
\]
\end{lemma}
\begin{proof}
The assertion amounts to proving that $\frk(\lam) \cap \ol{\frn}(\lam)\cap (1+N)\frp(\lam)(1+N)^{-1}$ is empty.
If this were not the case, then there would in particular be an element 
$Y \in \ol{\frn}(\lam)$ such that $[\Ad(1+N)^{-1}]Y \in \frp(\lam)$.
But $[\Ad(1+N)^{-1}Y]$ is a sum of $\ad(\lam)$ eigenvectors with strictly negative eigenvalues, and so cannot
belong to $\frp(\lam)$.
\end{proof}

\medskip

\noindent The previous lemma shows how $K(\lam) \cap \ol{N}(\lam)$ acts on the image of
$\Phi$; the next examines how $K(\lam) \cap P(\lam)$ acts.

\begin{lemma}  
\label{l:p} Fix
$N \in \frgone(\lam)$.
Then $L(\lam)\cdot\Phi(N)$ is dense in $(K(\lam) \cap P(\lam))\cdot \Phi(N)$.
\end{lemma}
\begin{proof}
The main result of \cite{zel} (due independently to Lusztig) asserts
that $L(\lam)\cdot \Phi(N)$ is dense in $P(\lam)\cdot\Phi(N)$.
Since $L(\lam) \subset K(\lam)$, the lemma follows.
\end{proof}

\medskip

\begin{proof}[Proof of Theorem \ref{t:geom}]
Let $Y$ denote the union of $K(\lam)$ orbits in $X(\lam) =G(\lam)/P(\lam)$
which meet the closure (in $X(\lam)$) of $\Phi(\frgone(\lam))$.
Set
\[
Y_\circ := (K(\lam) \cap\ol{N}(\lam))\cdot \Phi(\frgone(\lam)).
\]
Let $A = \frk(\lam) \cap \ol{\frn}(\lam)$, an affine space.
Lemma \ref{l:nbar} implies that
\begin{align*}
\varphi \; : \;A \times \frgone(\lam) &\longrightarrow Y_\circ \\
(a,N) &\longrightarrow \exp(a)\cdot \Phi(N)
\end{align*}
is an isomorphism.  Lemma \ref{l:p} implies $Y_\circ$ is dense in $Y$.
We next observe that analogous facts also hold on the level of the
relevant stratifications leading to the definitions of $M_\bbR^g$ and $M_\bbH^g$
in Theorem \ref{t:geom}.

Fix an orbit $Q$ in the stratification of $\frgone(\lam)$
by $L(\lam)$ orbits.  
Then, as above, the previous two lemmas
imply that $\Psi(Q)$ is the unique stratum in the stratification of $Y$
by $K(\lam)$ orbits such that
$\varphi(A \times Q) \simeq A \times Q$ is dense in $\Psi(Q)$; and, moreover, all
strata in $Y$ arise in this way.  
Thus the topological properties of intersection
homology imply the assertion of Theorem \ref{t:geom}.
\end{proof}

\medskip
With notation as in Sections \ref{ssec:glclass} and \ref{ssec:hclass},
define the ``pullback'' of $\Psi^g$, 
\begin{equation}
\label{e:Psi}
\Psi \; : \; \PR_n \longrightarrow \PH_n \cup \{0\},
\end{equation}
as follows.  Fix $\gamma\in \PR_n$.  If there exists
$\gamma'\in \PH_n$ such
that
\[
\Psi^g(d_\bbH(\gamma')) = d_\R(\gamma),
\]
then $\gamma'$ is unique and we set
\[
\Psi(\gamma) = \gamma'.
\]
If no such $\gamma'$ exists, set $\Psi(\gamma) = 0$.
Theorems \ref{t:glgeomclass} and \ref{t:hgeomclass}
immediately give the following corollary.

\begin{corollary}
\label{c:geom}
Suppose $\Psi(\gamma)  \neq 0$ 
and $\Psi(\eta) \neq 0$.  Then
\[
M_\R(\gamma,\eta) = M_\bbH(\Psi(\gamma),\Psi(\eta)).
\]
\end{corollary}

\medskip
The corollary has the following consequence. 
Suppose $\caF$ were an exact functor from $\HC_n$ to $\H_n$
such that
\beq
\label{e:caF}
\caF(\std_\R(\gamma)) = \std_\bbH(\Psi(\gamma))
\eeq
if $\Psi(\gamma) \neq 0$ and $\caF(\std_\bbR(\gamma)) = 0$ otherwise.
Then Corollary \ref{c:geom} immediately allows us to conclude
\beq
\label{e:caF2}
\caF(\irr_\bbR(\gamma)) = \irr_\bbH(\Psi(\gamma))
\eeq
if $\Psi(\gamma) \neq 0$ and $\caF(\irr_\bbR(\gamma)) = 0$ otherwise.

In the next two sections we will prove that the functor $F_{n,n}$ of 
the introduction satisfies \eqref{e:caF}, and hence \eqref{e:caF2},
but only when restricted to those parameters $\gam \in \PH_n$
of level at least $n$ (in the sense of Section \ref{ssec:lev}).

\section{images of standard modules}
\label{s:std}

In this section we define the functor $F_{n,k}:\HC_n\to \H_k$ 
carefully and
compute the images of certain standard modules
(those of level $\ge k$ in the language of Section \ref{ssec:lev}). 

\subsection{The functor $F_{n,k}$} 
\label{s:fnk}
For
the computations below it will
be convenient to introduce coordinates. In this paragraph
only, let $a$ denote $n$ or $k$.
Recall that we have identified $\frh$ and $\frh^*$ using the
trace form.  We further fix an isomorphism to $\bbC^a$ so
the pairing $\langle\cdot,\cdot\rangle$ of $\frh$ and $\frh^*$ becomes the usual dot product
in $\C^a$. We set 
\[
\text{$\al_i=(\underbrace{0,\dots 0}_{i-1},1,-1,0\dots,0)\in\C^a,$ $1\le i\le
a-1$,}
\]
identify the simple roots of $\frh$ in $\frg$ with $\{\al_1,\dots, \al_{a-1}\}$,
and set
\[
\text{
$\epsilon_j=(\underbrace{0,\dots,0}_{j-1},1,0,\dots,0),$ $1\le j\le a$
}
\] 
 In these coordinates, we have
$\rho=(\frac{a-1}2,\frac{a-3}2,\dots,-\frac{a-1}2).$ The 
reflections in the simple roots $\al_i$ will be denoted by $s_i.$

Given an object $X$ in $\HC_n$, consider
\[
\Hom_{K_\R}(\triv,(X\otimes\sgn)\otimes V^{\otimes k}),
\]
where $V=\C^n$ is the standard representation of $\frg$.  As in \cite[2.2]{AS},
we now define an action of $\bbH_k$ on $F_{n,k}(X)$, thus obtaining
an exact covariant functor
\[
F_{n,k}\; : \; \HC_n \longrightarrow \H_k.
\]
To begin, let
$B$ and $B^*$ denote two orthonormal bases
of $\frg=\mathfrak{gl}(n)$ (with respect to the trace form inner product $(A,B) = tr(AB)$).
Assume further that  they are dual to each other
in the sense that there is a bijection $B\to B^*$ such that the corresponding matrix
representation of $(\cdot, \cdot)$ is the identity.  For $E \in B$, write $E^*$ for
its image in $B^*$.
For $0\le i<j\le k$ define the operator
$\Omega_{i,j}\in  \End((X\otimes \sgn)\otimes V^{\otimes k})$ 
\begin{equation}
\Omega_{i,j}=\sum_{E\in B}(E)_i\otimes (E^*)_j,
\end{equation}
where, for simplicity of notation, we abbreviated the operator $1^{\otimes i}\otimes E\otimes
1^{\otimes j-i-1}\otimes E^*\otimes 1^{\otimes k-j}$ by
$(E)_i\otimes (E^*)_j.$
The definition of $\Omega_{i,j}$ does not depend on the choice of orthonormal
bases. Moreover if $1\le i<j\le k$, notice that $\Omega_{i,j}$
acts
by permuting the factors of $V^{\otimes k}$.

Consider the map $\Theta$ from $\bbH_k$ to $ \End((M\otimes \sgn)\otimes V^{\otimes k})$ 
defined by
\begin{equation}\label{theta}
\Theta(s_i)=-\Omega_{i,i+1},\ 1\leq i\leq k, \qquad \Theta(\epsilon_\ell)=\sum_{0\le
  l<\ell}\Omega_{l,\ell}+\frac {n-1}2,\ 1\le l\le k.
\end{equation}
It is not difficult to verify this definition respects the
commutation relation in $\bbH_k.$ It is also easy to see that the action
of $\bbH_k$ defined by $\Theta$ commutes with the $\frg$-action and the diagonal $K_\R$-action.
Thus we obtain the desired action of $\bbH_k$.

\begin{remark}
\label{r:fine}
{It is interesting to replace the trivial $K_\R$ type in the definition of
$F_{n,k}$ with other $K_\R$ types which 
are fine in the sense of Vogan (e.g.~\cite[Definition 4.3.9]{vogan:green}).
We shall return to this elsewhere.}
\end{remark}

\subsection{A notion of level for $\HC_n$}
\label{ssec:lev}
In order to state the main computation of $F_{n,k}$ applied to 
standard modules (Theorem \ref{t:stdcomp}), 
we must consider finer structure on the set $\PR_n$ of Section \ref{ssec:glclass}.
We define
\beq
\label{e:lev}
\lev \; : \; \PR_n \lra \bbZ^{\geq 0}
\eeq
as follows.  
Fix a (representative of a) parameter
$\gamma = (\delta, P_\R) \in \PR_n$.  Write $P_\R = L_\R N_\R$ and
\[
L_\R=\GL(n_1,\R)\times\dots\GL(n_r,\R) \qquad \delta=\delta_1\boxtimes\dots\boxtimes\delta_r.
\]
Define  
\begin{equation}\label{e:lev2}
\text{lev}(\gamma):=\sum_{i=1}^r \text{lev}(\delta_i), \text{ where } \text{lev}(\delta_i)=\left\{
\begin{aligned}&1,&\text{if }\delta_i=\delta(\triv,\nu_i) \text{ (as in (\ref{e:gl1ds}))}\\
&0,&\text{if }\delta_i=\delta(\sgn,\nu_i) \text{ (as in (\ref{e:gl1ds}))}\\
&l_i,&\text{if }\delta_i=\delta(l_i,\nu_i) \text{ (as in (\ref{e:gl2ds}))}\\
\end{aligned}\right.
\end{equation}
Clearly this is well-defined independent of the choice of representative for $\gamma$.
Informally, $\lev(\gamma)$ is a kind of measure of the size of the lowest $L_\R \cap K_\R$ type
of $\delta$, and thus the lowest $K_\R$ type of $\std_\R(\gamma)$ or $\irr_\R(\gamma)$.
(This can be made precise but we have no occasion to do so here.)

For $l \geq 0$, define
\[
\PR_{n,l} = \lev^{-1}(l),
\]
the parameters of level exactly $l$,
and $\PR_{n,\geq l} = \cup_{m \geq l} \PR_{n,m}$, the parameters of
level at least $l$.  We say $\std_\R(\gamma)$ or $\irr_\R(\gamma)$ is of level
$l$ if $\gamma$ is.
Define $\HC_{n,\geq l}$ to be the full subcategory of $\HC_n$ whose objects are subquotients
of standard modules of level at least $l$.  

{
\begin{example}
\label{ex:sph}
Suppose $\std_\R(\gamma)$ is a spherical (minimal) principal series.  Then $\gamma = (B_\R,\delta)$
for a Borel subgroup $B_\R$ and $\delta = \delta_1 \boxtimes \cdots \boxtimes \delta_n$ with
each $\delta_i$ of the form $\delta(\mathsf 1, \nu_i)$.  Thus $\gamma$ has level $n$, and
every subquotient of a spherical principal series for $G_\R$ is an object
in $\HC_{n,\geq n}$.
\end{example}

\medskip

Though we make no essential use of it, we include the following result
as indication that level is a reasonable notion.
}
\begin{prop}
\label{p:level}
Any irreducible object in $\HC_{n,\geq l}$ is
of the form $\irr_\R(\gamma)$ for $\gamma \in \PR_{n,\geq l}$.  More precisely,
every irreducible subquotient of a standard module of level $l$ has level at least $l$.
\end{prop}
\begin{proof}[Sketch]
By explicitly examining the Bruhat $\caG$-order of \cite[Definition 5.8]{v:ic3}, 
one may verify the proposition directly.
\end{proof}

We next define a map
\begin{equation}
\label{e:gammank}
\Gamma_{n,k} \; : \; \PR_{n,\geq k} \rightarrow \PH_k \cup \{0\}.
\end{equation}
as follows.
If $\lev(\gamma) > k$, set $\Gamma_{n,k}(\gamma) = 0$.
Next assume $\lev(\gamma) = k$ and fix a representative $(P_\R,\delta)$
of $\gamma$ with $\delta = \delta_1 \boxtimes \cdots \boxtimes \delta_r$. 
We define a representative of $\Gamma_{n,k}(\gamma)$ as follows.
Let $\frp = \frp_\lev(\gamma)$ denote the block upper triangular parabolic subalgebra
of $\frg\frl(k,\C)$ with Levi factor $\frl = \frl_\lev(\gamma)$ consisting of 
(ordered) diagonal blocks of size $\lev(\delta_1),
\dots \lev(\delta_r)$ (with notation as in \eqref{e:lev2}).  
Write $\bbH_P$
for the corresponding subalgebra of $\bbH_k$.
Let $t$ denote the number of $\GL(2,\R)$ factors in $P_\R$
and $m$ denote the number of $\GL(1,\R)$ factors in $P_\R$ whose
corresponding relative discrete series in $\delta$ is trivial when
restricted to $O(1)$.
 Thus there are $m+t$ simple factors in 
$\frl$.  List, in order, the corresponding relative discrete series in $\delta$
as $\delta_1', \dots, \delta_{m+t}'$.  Write $\nu_i'$ for the central character
of $\delta_i'$.  
Let $\delta^\bbH$ be the relative discrete series  $\delta^\bbH_1
\boxtimes \cdots \boxtimes \delta^\bbH_{m+t}$ of $\bbH_P$ with
$\delta^\bbH_i = \St \otimes \bbC_{\nu_i'}$.  Then $(\bbH_P,
\delta^\bbH)$ specifies an element of $\PH_k$, which we
define to be $\Gamma_{n,k}(\gamma)$.  For later use
we note that the central character of
$\std_\bbH(\Gamma_{n,k}(\gamma))$
is 
\begin{equation}\label{e:cc}
\left ( -\rho\left (\frg\frl(\lev(\delta'_1)) \right) + \nu_1' \; \bigr | \; \cdots  \cdots \; \bigr | \;  
-\rho\left (\frg\frl(\lev(\delta'_{m+t})) \right) + \nu_{m+t}' \right),
\end{equation}
where the vertical lines denote concatenation.

\begin{remark}
\label{r:notdom}
Note the identical construction could be made for parameters $(P_\R,\delta)$
not satisfying the dominance hypothesis of \eqref{e:dom}.  The resulting
pair $(\bbH_P,\delta^\bbH)$ would then be well-defined but need not
satisfy the dominance of \eqref{e:hdom}.
\end{remark}

With the convention that $\std(0) = 0$, the main result of this section is as follows.

\begin{theorem}
\label{t:stdcomp}
Recall the map $\Gamma_{n,k}$ of  \eqref{e:gammank} .  Then
\[
F_{n,k}\left ( \std_\R(\gamma)\right ) = \std_\bbH\left (\Gamma_{n,k}(\gamma)\right )
\]
as $\bbH_k$ modules
for all $\gamma \in \PR_{n,\geq k}$.
\end{theorem}

\begin{remark}
\label{r:notdom2}
In fact the proof below shows that induced modules which fail to satisfy \eqref{e:dom}
are also mapped to induced modules (or zero).  The correspondence of
parameters is given as in Remark \ref{r:notdom}.
\end{remark}

\subsection{Proof of Theorem \ref{t:stdcomp}}
We first prove the theorem on the level of vector spaces (Lemma \ref{l:dim}),
then on the level of $\bbC[W_k]$ modules (Lemma \ref{l:Wstr}), and then locate
an appropriate cyclic vector which allows us to deduce the $\bbH_k$-module
statement.

\begin{lemma}\label{l:dim}
Fix a representative $(P_\R, \delta)$ of $\gamma \in \PR_n$, write $\delta = \delta_1 \boxtimes \cdots
\boxtimes \delta_r$,  and consider $\std_\R(\gamma)$
as in Section \ref{ssec:glclass}.  Recall the level maps of \eqref{e:lev}
and \eqref{e:lev2}.  
Then we have $F_{n,k}(\std_\R(\gamma))\neq 0$ only if
  $\text{lev}(\gamma)\le k$. Moreover, if $\text{lev}(\gamma)=k,$ then
  \[
  \dim F_{n,k}(\std_\R(\gamma))=\frac{k!}{\prod_{i=1}^r\text{lev}(\delta_i)!}.
  \]
Thus, with notation as in Theorem \ref{t:stdcomp},
\[
\dim F_{n,k}\left ( \std_\R(\gamma)\right ) = \dim \std_\bbH\left (\Gamma_{n,k}(\gamma)\right ).
\]
\end{lemma}

\begin{proof} 
We have the following string of isomorphisms:
\begin{equation}
\aligned
F(\std_\R(\gamma))&=
\Hom_{K_\R}\left(\triv_{K_\R},
\Res_{G_\R}^{K_\R}\left(\Ind_{P_\R}^{G_\R}(\delta)\otimes\sgn\otimes V^{\otimes k}\right)\right)\\
&= \Hom_{K_\R}\left(\sgn_{K_\R},
\Res_{G_\R}^{K_\R}\left(\Ind_{P_\R}^{G_\R}(\delta)\otimes V^{\otimes k}\right)\right)\\
&= \Hom_{K_\R}\left(\sgn_{K_\R},
\Res_{G_\R}^{K_\R}\left(\Ind_{P_\R}^{G_\R}
(\delta \otimes V^{\otimes k})\right)\right )&\text{(by a Mackey isomorphism)}\\
&=\Hom_{K_\R}\left (\sgn_{K_\R},\Ind_{K_\R\cap L_\R}^{K_\R}(\delta\otimes V^{\otimes k})\right )
&\text{(restriction to $K_\R$)}\\
&=\Hom_{K_\R\cap L_\R}(\sgn_{K_\R\cap L_\R},\delta\otimes \Res^{K_\R}_{K_\R\cap
  L_\R}V^{\otimes k}) &\text{(Frobenius reciprocity)}\\
&=\Hom_{K_\R\cap L_\R}(\delta^*\otimes \sgn_{K_\R\cap L_\R}, \Res^{K_\R}_{K_\R\cap
  L_\R}V^{\otimes k}) .
\endaligned\label{e:Frob}
\end{equation}
Note that $K_\R\cap L_\R$ is a product of $O(2)$ and $O(1)$ factors. So
we need to understand $V^{\otimes k}$ as an $O(2)^s \times O(1)^{n-2s}$
module.  Recall the notation $V(j)$ for $O(2)$ types from Section \ref{ssec:glclass}.
We have (as a representation of $O(2)^s \times O(1)^{n-2s}$):
\begin{align}
\label{e:stdrep}
V \simeq &(\bigoplus_{i=1}^s \bbC \boxtimes \cdots \boxtimes
\stackrel{ith}{V(1)} \boxtimes \bbC \boxtimes \cdots \boxtimes \bbC) \oplus (\bigoplus_{j=1}^{n-2s} \bbC \boxtimes \cdots \boxtimes
 \stackrel{(j+2s)th}{\sgn} \boxtimes \bbC \boxtimes \cdots \boxtimes \bbC).
\end{align}
Notice that 
\begin{align*}
&V(1)^{\otimes p} = V(p) + \binom{p}{1} V(p-2) + \dots + \binom{p}{p'}V(1),
& \text{if $p = 2p' + 1$ is odd} \\
&V(1)^{\otimes p} = V(p) + \binom{p}{1} V(p-2) + \dots + \frac12\binom{p}{p'}V(0)
&\text{if $p=2p'$ is even.}
\end{align*}
We have (up to permutation of the factors):
\begin{equation*}
 \delta^*\bigr |_{K_\R\cap L_\R} = V(l_1) \boxtimes \cdots \boxtimes V(l_t)
 \boxtimes \triv^{\boxtimes (r-t-m)}\boxtimes \sgn^{\boxtimes m} 
 \; + \; \text{higher terms}
 \end{equation*}
where $\sum_{i=1}^t l_i+m=\lev(\gamma)$ and the higher terms involve larger
$L_\R \cap K_\R$ types. 
Using the rules for the tensor powers of $V(1)$ given above, we
see that the only way  $(\delta^*\otimes\sgn)\bigr |_{K_\R\cap L_\R}$
can appear in $V^{\otimes k}$ is if the $i$th $V(1)$ factor in \eqref{e:stdrep} contributes
$l_i$ times {\em and} there are $m$ distinct contributions of the
$j+2s$ $\sgn$ factors. This implies that if $\lev(\gamma)> k,$ then there can be no overlap between
$\delta^*$ and $\Res^{K_\R}_{K_\R\cap L_\R}V^{\otimes k},$ and therefore $F_{n,k}(\std_\R(\gamma))=0.$

Now for the second part assume that $\lev(\gamma)=k$ exactly. 
Then the dimension of the image $F_{n,k}(\std_\R(\gamma))$ is the coefficient
of 
$X_1^{l_1}\cdots X_t^{l_t}y_1\dots y_m$
in 
$(X_1 + \dots + X_t + y_1 + \dots + y_m)^k$,
namely
$\frac{k!}{l_1! \cdots l_t!},$ as claimed.
\end{proof}
\smallskip

\begin{remark}
The proof did not use the dominance of \eqref{e:dom} anywhere.  So the 
result holds for more general standard modules (not just those with Langlands
quotients).
\end{remark}

\medskip

For calculations below, we
will need to specify a basis
of $F_{n,k}(\gamma)$ for $\gamma$ of level $k$.
As before let $(P_\R,\delta)$ denote a representative of $\gamma$.
We will use the chain of isomorphisms of \eqref{e:Frob}
and specify a basis of $F_{n,k}(\gamma)$ by instead specifying
a basis of
\beq
\label{e:homspace}
\Hom_{K_\R\cap L_\R}(\sgn_{K_\R\cap L_\R},\delta\otimes
V^{\otimes n}).
\eeq 
We need some notation. For $\mathfrak{gl}(2,\C),$ we will need the following basis:
\begin{align}
H=\left(\begin{matrix}0&-i\\i&0\end{matrix}\right),\ E_+=\frac
12\left(\begin{matrix}1&i\\i&-1\end{matrix}\right),\ E_-=\frac 12\left(\begin{matrix}1&-i\\-i&-1\end{matrix}\right),\ Z=\left(\begin{matrix}1&0\\0&1\end{matrix}\right). 
\end{align}   
Then $\{E_+,H,E_-\}$ form a Lie triple, and $Z$ generates the
center. Let $f_\pm$ be the eigenvectors of $H$ on $\C^2$:
$f_+=\left(\begin{matrix}1\\i\end{matrix}\right),$
  $f_-=\left(\begin{matrix}i\\1\end{matrix}\right).$ Then we have
\begin{align*}
&H\cdot f_+=f_+, \qquad E_+\cdot f_+=0,  \quad E_-\cdot f_+=i f_-,\\
&H\cdot f_-=-f_-, \quad E_-\cdot f_-=0, \quad E_+\cdot f_-=-i f_+.
\end{align*} 
Let
$w_{\pm}^{(i)}$ denote highest $\SO(2)$ weight vectors (of respective
weights $\pm l_i$) in the two-dimensional lowest $K$-type of
the relative discrete series
$\delta(l_i,\nu_i)$ of $\GL(2,\R).$   Assume further that $w_+^{(i)}$ and $w_-^{(i)}$
are interchanged by the element $\left ( \begin{matrix}0 & 1\\ 1 & 0\end{matrix}
\right )$ in $O(2)$.
These vectors satisfy 
\begin{align*}
E_+w^{(i)}_-=0,\quad E_-w^{(i)}_+=0,\quad  Hw^{(i)}_\pm=\pm l_i w^{(i)}_\pm.
\end{align*}
We need to embed these kinds of $\frg\frl(2,\C)$ considerations inside $\frg\frl(n,\C)$,
and this involves some extra notation. 

Fix $(P_\R, \delta)$ of level $k$ and, as usual,
let $\delta$ and $\delta_i$ denote the representation
spaces of the relevant relative discrete series.  For each $i$ such that $n_i = 1$,  fix nonzero vectors in $x_\pm^{(i)} \in \delta_i$;
for each $i$ such that $n_i = 2$,  set $x_\pm^{(i)} = w_\pm^{(i)}$.
Set  
\[
x_\pm=x_\pm^{(1)}\boxtimes\dots \boxtimes x_\pm^{(r)}\in\delta.
\]
Next let $e_1, \dots, e_n$ denote the basis of $V$ we have been
implicitly invoking.  For an index $i$ such that $n_i =1$, 
let $e_\pos{i}$ 
denote the basis element corresponding to the position of $i$th
component of the Levi factor of $P_\R$.  For an index $i$ such that
$n_i = 2$, let $e_{\pos{i}-1}, e_{\pos{i} }$ denote the pair of basis
elements corresponding to the position of the $i$th component of the
Levi factor of $P_\R$. 
For an index $i$ such that
$n_i = 2$ and $z = \left ( \begin{matrix}a\\b\end{matrix}\right) \in
\bbC^2$
let $z^{(\pos{i})}$ denote $ae_{\pos{i}-1} + be_{\pos{i}}$ in $V$.
Finally for an element $X$ of $\frg\frl(2,\C)$
and an index $i$ such that $n_i = 2$, let $X^\pos{i}$ denote the image
of $X$ in $\frg\frl(n,\C)$ under the inclusion of $\frg\frl(2,\C)$ into the $i$th component
complexified Lie algebra of the Levi factor of $P_\R$.   
Set
\[
u_\mp^{(i)}=\left\{\begin{aligned}&\underbrace{f_\mp^{(\text{pos}(i))}\otimes\dots\otimes
  f_\mp^{(\text{pos}(i))}}_{l_i}, \ \text{ if } n_i=2,\\
  &u_\mp^{(i)}=e_{\text{pos}(i)},\ \text{ if } n_i=1\text{ and }
  \delta_i=\delta(\triv,\nu_i)\end{aligned} \right.\\
\]
and
\[
u_\mp= \bigotimes_i u_\mp^{(i)} \in V^{\otimes k},
\]
where the tensor is over indices $i$ such that $n_i =2$ or $n_i =1$ and $\delta_i$ is trivial when
restricted to $O(1)$.  (This tensor is in $V^{\otimes k}$ since we have assumed $\gamma$ to 
be of level $k$.)
Finally set
\beq
\label{eigvec}
v_\gamma=x_+\otimes u_-+x_-\otimes u_+\in(\delta\otimes V^{\otimes k}).
\eeq
Then the basis of the space in \eqref{e:homspace}
is formed of the orbit of $v_\gamma$ under the $W_k$ action permuting
the $V^{\otimes k}$ factors of $v_\gamma$.
Clearly the stabilizer of $v_\gamma$ is $\prod_{i=1}W_{\lev(\delta_i)}$.
Together with the definition of the standard $\bbH$ modules, we obtain the following
lemma.

\begin{lemma}\label{l:Wstr}
In the setting of Theorem \ref{t:stdcomp}, we
have  
\[
F_{n,k}\left ( \std_\R(\gamma)\right ) = \std_\bbH\left (\Gamma_{n,k}(\gamma)\right )
\]
as modules for the group algebra $\bbC[W_k]$.
\end{lemma}
\qed

We now turn to the proof of Hecke algebra action in Theorem \ref{t:stdcomp}. 
 Assume $\gamma$ has
level $k$, so $\Gamma_{n,k}(\gamma) \neq 0$ by Lemma \ref{l:dim}.
From the definition of standard modules given in Section \ref{ssec:hclass},
the $\bbH_k$-module $\std_\bbH(\Gamma_{n,k}(\gamma))$ 
is generated under $W_k$ by a cyclic vector
$\mathsf v$ which transforms like the sign representation of
$W(\frl_\lev(\gamma)),$ and which is a common eigenvector for all generators
$\epsilon_i$,
\begin{equation}
\label{e:eigval}
\epsilon_i\cdot \mathsf v =\bigr \langle\epsilon_i,\chi(\Gamma_{n,k}(\gamma))\bigr\rangle.
\end{equation}
where $\chi(\Gamma_{n,k}(\gamma))$ denotes the central character of $\std_\bbH(\Gamma_{n,k}(\gamma))$
(as in \eqref{e:cc}).
In light of Lemma \ref{l:Wstr}, it is sufficient to find an element
of $F_{n,k}(\std_\R(\gamma))$ which transforms like the sign representation
of  $W(\frl_\lev(\gamma))$ and which is a common eigenvector for $\Theta(\epsilon_i)$ with the same
eigenvalues as in \eqref{e:eigval}.

The correct cyclic vector in $F_{n,k}(\std_\R(\gamma))$ 
is the one corresponding to $v_\gamma$ (from \eqref{eigvec})
under the chain of isomorphisms in \eqref{e:Frob}.  In fact, since
$F_{n,k}(\std_\R(\gamma))$ is isomorphic to the $K_\R$ 
invariants (for the diagonal action) in
\begin{align}
\nonumber
\Ind_{P_\R}^{G_\R}(\delta \otimes \sgn) \otimes V^{\otimes k} 
&\simeq 
\Ind_{P_\R}^{G_\R}(\delta \otimes \sgn) \otimes
\Ind_{G_\R}^{G_\R}(V)^{\otimes k} \\ \label{e:kfold}
&
\simeq
\Ind_{P_\R \times G_\R \times \cdots \times G_\R }^{G^{k+1}_\R}\left ( (\delta \otimes \sgn) \otimes
V \otimes \cdots \otimes V \right ), 
\end{align}
we will find it convenient to specify $f_\gamma$ as a $K_\R$ invariant element of the latter space,
i.e.~ as a function on the $(k+1)$-fold product $G_\R^{k+1}$.
The advantage of this point of view is that the left-translation
action of $G_\R^{k+1}$
is transparent, and that will of course be helpful below.
To unwind the definition of $f_\gamma$ from $v_\gamma$,
first consider the induced representation
\[
\Ind_{P_\R}^{G_\R}(\delta \otimes \sgn \otimes V^{\otimes k})
\]
of $\delta \otimes V^{\otimes k}$-valued functions on $G_\R$.
Using the decomposition $G_\R=K_\R P_\R=K_\R M_\R A_\R N_\R$,
set
\[
f^\Delta_\gamma(kman) = (man)^{-1}\cdot v_\gamma;
\]
the action of $(man)^{-1}$ is the diagonal one.  By construction, this is well-defined and $K_\R$
invariant.  Finally define the function $f_\gamma$ in the space of \eqref{e:kfold}
as
\beq\label{fgamma}
f_\gamma(x_0,x_1,\dots, x_n) = \pi_1(x_1^{-1}x_0)\cdots\pi_n(x_n^{-1}x_0)f^\Delta_\gamma(x_0);
\eeq
here $\pi_i(g)$ denote the action of $g \in G_\R$ in
the $(i +1)$st factor of the tensor product $\delta \otimes V^{\otimes k}$.

Then $f_\gamma$ is well-defined and $K_\R$ invariant.  It is completely characterized by
the value of $f^\Delta_\gamma$ at the identity $\mathsf 1$, i.e.~$v_\gamma$.
It is clear that $f_\gamma$ transforms like the sign representation 
for $\Theta(s)$ with $s\in W(\frl_\lev(\gamma))$ (since $v_\gamma$ does).
So it remains to compute the action of
$\Theta(\epsilon_\ell)=\sum_{0\le l\le \ell}\Omega_{l,\ell}$ on $f_\gamma$
and see that it is a simultaneously eigenvector with eigenvalues
as in \eqref{e:eigval}. 

\medskip

We choose, as we may, a
convenient basis for $\frg=\mathfrak{gl}(n)$ and define $\Omega_{0,\ell}$ with respect to
it. Recall the notation $\text{pos}(i)$ introduced before (\ref{eigvec}).  Write $\{E_{i,j}\}_{1\leq i,j\leq n}$ for
the usual basis of $\frg$.  The basis we choose consists of the following elements:

\begin{enumerate}
\item[$\bullet$] 
\begin{equation}\label{gl2basis}\{\frac 1{\sqrt 2} E_+^{(\text{pos}(i))},\frac 1{\sqrt 2} E_-^{(\text{pos}(i))},\frac 1{\sqrt 2}H^{(\text{pos}(i))},\frac
1{\sqrt 2} Z^{(\text{pos}(i))}\},\end{equation} {
for indices $i$ such that $n_i=2$;

\item[$\bullet$]
{
$E_{\pos{i},\pos{i}}$ for indices $i$ such that $n_i =1$; and}}

\item[$\bullet$] $E_{i,j}\in\frn\oplus\overline\frn,$ where $\frn$
  denotes the complexified Lie algebra of $N_\R$, and $\overline\frn$ denoted the
  opposite nilradical.
\end{enumerate}

Recall the notation $\pi_i(~)$ for the action of $g\in G_\R$ in the
$(i+1)$st factor of $\delta\otimes V^{\otimes k}.$ We use the same
notation for the corresponding action of an element $E\in\frg.$

The calculation is divided into three parts, depending if the element
$E\in \frg$ in the term $(E)_0\otimes (E^*)_\ell$ of $\Omega_{0,\ell}$ belongs
to $\frn,$ $\overline\frn,$ or $\frl.$

\begin{lemma}\label{l:nil} Assume that $E,F$ are elements of $\frp=\frl+\frn.$ Then, we have 
\begin{equation}
[((E)_0\otimes
  (F)_\ell)f_\gamma](\mathsf 1)=\pi_0(E)\pi_\ell(F)\:v_\gamma.
\end{equation}
 In
  particular, if $E\in \frn,$ then  $((E)_0\otimes
  (F)_\ell)f_\gamma=0.$

\end{lemma}

\begin{proof}
From (\ref{fgamma}), we have:
\begin{equation}
\aligned
((E)_0\otimes
  (F)_\ell)f_\gamma(\mathsf 1)&
=\left.{\dfrac
    {d^2}{du\:ds}}\right\vert_{u=s=0} \prod_{{l=1,l\neq
        \ell}}^k\pi_l(e^{-uE})\pi_\ell(e^{sF}e^{-uE})
      f^\triangle(e^{-uE})\\
&=\left.{\dfrac
    {d^2}{du\:ds}}\right\vert_{u=s=0}~ \pi_\ell(e^{sF})
      \pi_0(e^{uE}) ~v_\gamma=\pi_0(E)\pi_\ell(F)~v_\gamma,
\endaligned\label{eq:diff}
\end{equation}
where we have used that
$f^\triangle(e^{-uE})=[\prod_{l=1}^k\pi_l(e^{uE})] f^\triangle(\mathsf
1),$ since $e^{-uE}\in N_\R.$

For the second claim, it is sufficient to notice that
if $E\in\frn,$ then $\pi_0(e^{uE})v_\gamma=v_\gamma.$
\end{proof}

Note that, in particular, Lemma \ref{l:nil} implies that  
\begin{equation}\label{eq:nil}
[\sum_{E\in\frn} (E)_0\otimes (E^*)_\ell]f_\gamma=0. 
\end{equation}
 To compute the action of the terms
$(E_{i,j})_0\otimes (E_{j,i})_\ell$ with $E_{i,j}\in \overline\frn,$
we will find the following calculation useful.

\begin{lemma}\label{l:practical} Assume that $E_{i,j}\in\overline\frn$. Then 
\begin{equation}\label{eq:practical}
[((E_{i,j})_0\otimes (E_{j,i})_\ell)f_\gamma](\mathsf
1)=[(-\pi_\ell(E_{j,j})+\pi_\ell(E_{j,i})\sum_{{l=1, l\neq \ell}}^k\pi_l(-E_{i,j}+E_{j,i}))]\:v_\gamma.
\end{equation}

\end{lemma}

\begin{proof}
Since we have $E_{j,i}\in \frn,$ from Lemma \ref{l:nil}, it follows
that $(E_{j,i})_0\otimes (E_{i,j})=0$ on $f_\gamma.$  Therefore, if we
set $H_{i,j}=E_{i,j}-E_{j,i},$ we see that 
$(E_{i,j})_0\otimes (E_{j,i})_\ell=(H_{i,j})_0\otimes
(E_{j,i})_\ell.$  Since $H_{i,j}\in\frk,$ the (equivalent) operator $(H_{i,j})_0\otimes
(E_{j,i})_\ell$
is easier to compute. By a computation similar to (\ref{eq:diff}), we find that 
\begin{equation*}
((H_{i,j})_0\otimes (E_{j,i})_\ell) f_\gamma(\mathsf 1)=\left.{\frac
    {d^2}{du\:ds}}\right\vert_{u=s=0} \prod_{{l=1,l\neq
        \ell}}^k\pi_l(e^{-uH_{i,j}})\pi_\ell(e^{sE_{j,i}}e^{-uH_{i,j}})f^\triangle_\delta(e^{-uH_{i,j}}).
\end{equation*}
Notice that $f^\triangle_\delta(e^{-uH_{i,j}})=v_\gamma$, since
$e^{-uH_{i,j}}\in K_\R.$ Then the claim follows by applying the chain rule. 
\end{proof}

Before computing the action of the terms $\sum_{E_{i,j}\in\overline
  \frn} (E_{i,j})_0\otimes (E_{j,i})_\ell$ in $\Omega_{0,\ell}$ on
$f_\gamma,$ we need to establish more notation.

\begin{notation}\label{phi}For every $1\le\ell\le k,$ set $\phi(\ell)=i$, if the
  $\ell$-th factor of the tensor product $u_\mp$ from (\ref{eigvec})
  is  $f_\mp^{(\pos i)}$ or $e_{\pos i}$. Set also
  $\text{prec}(\ell)=\sum_{j=1}^{i-1}|u_\mp^{(j)}|,$ where
  $|u_\mp^{(j)}|$ is the (tensor) length of $u_\mp^{(j)}$.  Notice
  that 
  $\text{prec}(\ell)\leq \ell-1,$ and 
  if
  $n_{\phi(\ell)}=1,$ then $\text{prec}(\ell)=\ell-1$. 
\end{notation}

\begin{lemma}\label{l:nilbar} With the notation as in \ref{phi}, the action of $\displaystyle{\sum_{E_{i,j}\in\overline\frn}
  (E_{i,j})_0\otimes (E_{j,i})_\ell}$ on $f_\gamma$ equals:
\begin{equation*}
[{-\sum_{l=1}^{\text{prec}(\ell)}\Omega_{l,\ell}-n+p]~\Id},\quad\text{where
} p=\text{pos}(\phi(\ell)).
\end{equation*}
\end{lemma}

\begin{proof}
We use Lemma \ref{l:practical}. There are two cases, depending if
$n_{\phi(\ell)}=2$ or $1.$ Assume that $n_{\phi(\ell)}=2.$ In order
for a term $(E_{i,j})_0\otimes (E_{j,i})_\ell$ to act nontrivially,  it
is necessary that either $\pi_\ell(E_{j,j})$ or $\pi_\ell(E_{j,i})$
from (\ref{eq:practical}) act nontrivially on the $\ell$th factor of
$v_\gamma.$ With our notation, the $\ell$th factor of $v_\gamma$ is
$f_\mp^{(p)}.$ Recall that the convention is that $f_\mp^{(p)}$ is in
the $\C$-span of the vectors $e_{p-1},e_p.$ This implies that we must be in one of the following
two cases:

\begin{enumerate}
\item $j\in\{p-1,p\},$ and $p+1\le i\le n.$ Then
  $(E_{i,j})_0\otimes (E_{j,i})_\ell=-\pi_\ell(E_{j,j}),$ and so
  $$\sum_{j=p-1}^{p} (E_{i,j})_0\otimes (E_{j,i})_\ell=-\Id.$$
\item $i\in \{p-1,p\},$ and $1\le j\le p-2.$ Then we have
  $$(E_{i,j})_0\otimes (E_{j,i})_\ell=\pi_\ell(E_{j,i})\sum_{l=1,l\neq
  \ell}^{\text{prec}(\ell)}\pi_l(-E_{i,j}+E_{j,i}).$$ It is not hard to
  verify directly that
  $$\sum_{j=1}^{p-2}\sum_{i=p-1}^{p}(E_{i,j})_0\otimes (E_{j,i})_\ell=-\sum_{l=1}^{\text{prec}(\ell)}\Omega_{l,\ell}.$$
\end{enumerate}

In the second case ($n_{\phi(\ell)}=1$), the $\ell$th factor of $v_\gamma$ is
$e_p.$ Therefore, the only nonzero contributions come from:

\begin{enumerate}
\item $j=p$, and  $i\ge p+1.$ Then we have
  $(E_{i,p})_0\otimes (E_{p,i})_\ell=-\pi_\ell(E_{p,p})=-\Id.$
\item $i=p,$ and $j\le p-1.$ Then we have $$(E_{p,j})_0\otimes
  (E_{j,p})_\ell=\pi_\ell(E_{j,p})\sum_{l=1}^{\ell-1}\pi_l(-E_{p,j}).$$ When
  we sum over all $j,$ we get $$\sum_{j=1}^{\ell-1}(E_{p,j})_0\otimes
  (E_{j,p})_\ell=-\sum_{l=1}^{\ell-1}\Omega_{l,\ell}.$$
\end{enumerate}
 
\end{proof}

Next, we compute the action of the terms corresponding to $\frl.$

\begin{lemma}\label{l:levi} With the notation as in \ref{phi} and (\ref{e:cc}), the action of $\displaystyle{\sum_{E\in\frl}(E)_0\otimes
  (E^*)_\ell}$ on $f_\gamma$ equals:
\begin{equation*}
[-\frac{\text{lev}(\delta_p')}2+\nu'_{p}+\frac {n}2-p+1]~\Id,\quad\text{where
} p=\text{pos}(\phi(\ell)).
\end{equation*}
\end{lemma}

\begin{proof}
We use Lemma \ref{l:nil}. Assume first that $n_{\phi(\ell)}=2,$ so that the $\ell$th factor in
$u_\mp$ is $f_\mp^{(p)}.$ The only nontrivial actions could come from
$\{\frac 1{\sqrt 2}E_{\pm}^{(p)}, \frac 1{\sqrt 2}H^{(p)}, \frac
1{\sqrt 2}Z^{(p)}\}$ (notation as in (\ref{gl2basis})):

\begin{enumerate}
\item $$[(\frac
12 (H^{(p)})_0\otimes (H^{(p)})_\ell) f_\gamma](\mathsf 1)=-\frac
{\text{lev}(\delta_p')}2 v_\gamma,$$ since $H^{(p)}w_\pm^{(p)}=\pm
\frac{\text{lev}(\delta_p')}2 w_\pm$ and $H^{(p)} f_\mp^{(p)}=\mp
f_\mp^{(p)}$;
\item $$[(\frac
12 (Z^{(p)})_0\otimes (Z^{(p)})_\ell) f_\gamma](\mathsf 1)=(\nu'_p+\frac
{n}2-p+1) v_\gamma,$$ where $\frac n2-p+1$ is the $\rho$-shift in the
normalized induction in this case; 
\item $$[(\frac
12 (E_{\pm}^{(p)})_0\otimes (E_{\mp}^{(p)})_\ell) f_\gamma]=0,$$ since
$E_\pm w_{\mp}^{(p)}=0,$ and $E_\mp f_{\mp}^{(p)}=0.$
\end{enumerate}

In the case that $n_{\phi(\ell)}=1$, the only nontrivial action of
$f_\gamma$ is
that of 
$(E_{p,p})_0\otimes (E_{p,p})_\ell=[\nu_p'+\frac
{n-1}2-p+1]\,\Id$
where $\frac {n-1}2-p+1$ is the $\rho$-shift in this case. Recall also
that $\text{lev}(\delta_p')=1$ if $n_{\phi(\ell)}=1.$
\end{proof}

Finally we can compute the eigenvalue of $\Theta(\epsilon_\ell)$
on $f_\delta,$ which concludes the proof of Theorem \ref{t:stdcomp}.

\begin{proposition}  We have $$\Theta(\epsilon_\ell)f_\gamma=\langle
  \epsilon_\ell,\chi(\Gamma_{n,k}(\gamma))\rangle f_\gamma,\quad\text{ for all
   $1\le\ell\le n,$}$$ where the action of $\Theta(\epsilon_\ell)$ is defined
  in (\ref{theta}), $\chi(\Gamma_{n,k}(\gamma))$ is the central
  character from (\ref{e:cc}), and $f_\gamma$ is the eigenfunction (\ref{fgamma}).
\end{proposition}

\begin{proof}
By putting together (\ref{eq:nil}) and Lemmas 
\ref{l:nilbar}, \ref{l:levi}, we see that 
$$\Omega_{0,\ell}f_\gamma=-\sum_{l=1}^{\text{prec}(\ell)}\Omega_{l,\ell}-\frac{\text{lev}(\delta_p')-1}2-\frac{n-1}2,$$
where $p=\pos{\phi(\ell)}.$
From (\ref{theta}), we find then
$$\Theta(\epsilon_\ell)f_\gamma=-\frac{\text{lev}(\delta_p')-1}2+\sum_{
l=\text{prec}(\ell)+1}^{\ell} \Omega_{l,\ell} f_\gamma.$$
Now the claim follows by observing that
$\Omega_{l,\ell}f_\gamma=f_\gamma,$ for every $l$ such that
$\text{prec}(\ell)<l<\ell$, since in this case $\Omega_{l,\ell}$
permutes identical factors.
\end{proof}

\begin{example} 
\label{ex:basic}
We give some basic examples. 
\begin{enumerate}
\item  If $n=k$ and $\std_\R(\gamma)$ is the spherical minimal principal
  series of $\GL(n,\R)$, with spherical quotient, then
$\std_\bbH(\Gamma_{n,k}(\gamma))$ is the spherical principal series of
$\bbH_k$ with spherical quotient (and the
same central character as the infinitesimal character of $\std_\R(\gamma)$).

\item  Assume that $n=2m$ and $k=dm,$ $d\geq 2,$ 
  and define \begin{align*}\std_\R(\gamma)=\Ind_{P_\R}^{G_\R}\left({\delta\left(d,\frac{m-1}2\right)\boxtimes\dots\boxtimes \delta\left(d,-\frac{m-1}2\right)}\right).
\end{align*}
where $L_\R=GL(2,\R)^m.$
Then Theorem \ref{t:stdcomp} implies that
\begin{equation*}F_{2m,dm}(\std_\R(\gamma))=\std_\bbH(\Gamma_{2m,dm}(\gamma))=\bbH_{dm}\otimes_{\bbH_d^m}(\St\otimes\C_{\frac
  {m-1}2}\boxtimes\dots\boxtimes\St\otimes
  \C_{-\frac{m-1}2}).\end{equation*} 
 The
unique irreducible quotient of $\std_\R(\gamma)$ is a Speh representation
for $\GL(2m,\R)$,
while the unique irreducible quotient of $\std_\bbH(\Gamma_{2m,dm}(\gamma))$ is a
Tadi\'c representation for $\bbH_{dm}.$
\end{enumerate}
\end{example}

\section{images of irreducible modules}
\label{s:irr}

This section is devoted to the following result.
 \begin{theorem}
\label{t:irr}
Write $\Psi_{n,n}$ for restriction of the map $\Psi$ of \eqref{e:Psi}
to  the set of parameters
$\PR_{n,\geq n}$ for $\GL(n,\R)$ of level at least $n$ 
(Section \ref{ssec:lev}).  
Recall the map $\Gamma_{n,n}$ of
\eqref{e:gammank}.  Then
\[
\Phi_{n,n} = \Gamma_{n,n}.
\]
 \end{theorem}

\medskip

Arguing as at the end of Section \ref{s:geom}, and using the convention
that $\irr_\bbH(0) = 0$, we immediately obtain
our main result.

\begin{corollary}
\label{c:irrcomp}
If $X$ is an irreducible object in the category $\HC_{n,\geq n}$ Harish-Chandra modules for $\GL(n,\R)$
with level at least $n$ (Section \ref{ssec:lev}), 
then $F_{n,n}(X)$ is irreducible or zero.  More precisely
\[
F_{n,n}\left ( \irr_\R(\gamma)\right ) = \irr_\bbH\left (\Gamma_{n,n}(\gamma)\right ),
\]
which is nonzero if and only if $\gamma$ has level $n$.  All irreducible objects in $\H_n$ arise
in this way.  In particular, $F_{n,n}$ implements a bijection between the irreducible objects
in {$\HC_{n}$} of level $n$ and the irreducible objects in $\H_n$.
\end{corollary}
\qed

\begin{example} There are two extremes:

\begin{enumerate}
\item When $\irr_\R(\gamma)$ is the spherical quotient of the minimal
  principal series $\std_\R(\gamma),$ then
  $\irr_\bbH(\Gamma_{n,n}(\gamma))$ is the spherical quotient
  of the corresponding minimal principal series
  $\std_\bbH(\Gamma_{n,n}(\gamma))$ (see Example
  \ref{ex:basic}(1)). In particular, the trivial $G_\R$ representation is
  mapped to the trivial $\bbH_n$-module.

\item  Set $\std_\R(\gamma)=\Ind_{P_\R}^{G_\R}(
  \delta(\sgn,\frac
        {n-1}2-1)\boxtimes\dots\boxtimes\delta(n,0)\boxtimes\dots\boxtimes\delta(\sgn,{-\frac{n-1}2+1})),$
   where $L_\R=\GL(1,\R)^{\lceil\frac n2\rceil-1}\times\GL(2,\R)\times
   \GL(1,\R)^{\lfloor\frac n2\rfloor-1},$
  and $\delta(n,0)$ is inserted between
  $\delta(\sgn,{\frac 12})$ and $\delta(\sgn,{-\frac 12})$ when $n$
  is even, or $\delta(\sgn,0)$ and $\delta(\sgn,{-1}),$ if $n$ is
  odd. Then $F_{n,n}(\std_\R(\gamma))$ is the Steinberg $\bbH_n$
  module $\St$, and
  $F_{n,n}(\irr_\R(\gamma))=\irr_\bbH(\Gamma_{n,n}(\gamma))=\St.$ 
\end{enumerate}

\end{example}

\medskip

Since we have defined $\Gamma_{n,n}$ and $\Psi_{n,n}$ quite explicitly, the proof
of Theorem \ref{t:irr}, which we shall sketch in the remainder of this section,
 amounts to a rather unenlightening combinatorial verification. 
 First we note that it is obvious from the definitions that we may work with a
 fixed infinitesimal and central character $\lam$.  It is then not difficult to reduce to
 the case of integral $\lam$, and we impose that assumption henceforth.  (There is hard work involved in this reduction to the integral case, but it
 is buried in Theorems \ref{t:glgeomclass} and \ref{t:hgeomclass}
 and the references to \cite{abv} and \cite{lu:cls1}.) 
 After twisting by the center, there is no harm in assuming that $\lam$ consists of a (weakly) decreasing
 sequence of $n$ integers.
 
 The map $\Gamma_{n,n}$ is given very explicitly in \eqref{e:gammank}.  We need to
 be similarly explicit with the maps $d_\R$, $d_\bbH$, and $\Psi^g$ which go
 into the definition of $\Psi_{n,n}$.   We treat each of these individually, starting
 with $\Psi^g$.
 
 First we discuss the parameter space $\gPR_n(\lam)$.  {Recall that
 the only relevant parameters for us are the ones corresponding to
 the set of $K(\lam)$ orbits on $G(\lam)/P(\lam)$.}
 Since $\lam$
 is assumed to be integral, $G = G(\lam)$ and $K(\lam) \simeq 
 \GL(p,\C) \times \GL(q,\C)$ for $p+q =n$ (where $p$ is the number
 of even entries in $\lam$ and $q$ is the number of odd entries).  
 Let $B$ denote a Borel subgroup of $G$,
 and (since $\lam$ is fixed), write $K=K(\lam)$, $P=P(\lam)$, and $P=LU$
 to conserve notation.
 Write $\pi$ for the projection of $G/B$ to $G/P$.
 The orbits of $K$ on $G/P$ thus are parametrized
 by equivalence classes of $K$
 orbits on $G/B$ for the relation $Q \sim Q'$ if
 $\pi(Q) = \pi(Q')$.  The orbits of $K$ on $G/B$ are parametrized 
 by certain twisted involutions on which the equivalence relation
 is easy to read off (\cite{rs}).   We recall the combinatorics now.
 
The set of $K$ orbits on $G/B$ is parametrized by
involutions in
$S_{n}$ with signed fixed points of signature $(p,q)$; that is,
involutions in the symmetric group $S_{n}$
whose fixed points are labeled with
signs (either $+$ or $-$) so that half the number of non-fixed points plus
the number of $+$ signs is exactly $p$ (or, equivalently, half the number
of non-fixed points plus the number of $-$ signs is $q$).
Given such an involution $\sigma_\pm$, write $Q_{\sigma_\pm}$
for the corresponding orbit.
Identify the Weyl group of $P$ with $S_{m_1} \times \cdots \times
S_{m_r}$ inside $S_n$.  For a simple transposition $s \in W_P$ in the coordinates $i$ and $i+1$
and an involution with signed fixed points $\sigma_\pm$, define a new
involution with signed fixed points $s\cdot \sigma_\pm$ as follows: 
(1) if
the coordinates $i$ and $i+1$ of $\sigma_\pm$ are fixed points with opposite
signs, replace them by the transposition $s$ but make no other changes to
$\sigma_\pm$; (2) if
the coordinates $i$ and $i+1$ of $\sigma_\pm$ are fixed points with the same
sign or else are nonfixed points interchanged by $\sigma_\pm$, do nothing;
and (3) in all other cases, let $s\cdot \sigma_\pm$ be obtained from $\sigma_\pm$
by the obvious conjugation action of $s$ on involutions with signed fixed points.  
(See \cite[Section 2]{mt}, for instance,
for a more careful discussion.)  Then the equivalence relation $Q \sim Q'$
on $K$ orbits is generated by $Q_{\sigma_\pm} \sim Q_{s\cdot \sigma_\pm}$.

Correspondingly we introduce a combinatorial model  for $\gPH_n$.
Define a segment to be a finite increasing sequence of complex
numbers, such that any two consecutive terms differ by $1$.  A
multisegment is an ordered collection of segments. If $\tau$ is a multisegment,
define the support $\underline\tau$ of $\tau$ to be the set of all
elements (with multiplicity) of all the segments in the
multisegment. Set
\begin{equation}
\caM(\lambda)=\{\tau\text{
  multisegment}:\ \underline\tau=\lambda\text{ (up to permutation)}\}.
\end{equation}
If $\tau,\tau'\in\caM(\lambda),$ define 
$\tau\sim\tau'$ if $\tau$ and $\tau'$ have the same segments (in
different order). This is an equivalence relation on $\caM(\lambda)$,
whose classes we shall denote by $\caM_\circ(\lambda)$
Then there is a one-to-one correspondence (\cite{Ze})
\begin{equation}\label{eq:multiparam}
L(\lambda)\backslash\frg_{-1}(\lambda)\longleftrightarrow\caM_\circ(\lambda),
\end{equation}
and, as discussed in Section \ref{ssec:hgeom}, $\gPH_n$ identifies
with the orbits of $L(\lam)$ on $\frgone$. For example, if
  $\lambda=\rho$, there are
  $2^{n-1}$ multisegments in $\caM_\circ(\rho).$ In
  (\ref{eq:multiparam}), the zero $L(\rho)$-orbit is parameterized by
    the multisegment $\{\{\frac{n-1}2\},\dots,\{-\frac{n-1}2\}\}$, while the
    open $L(\rho)$-orbit is parameterized by
    $\{\{-\frac{n-1}2,\dots,\frac{n-1}2\}\}$.
 
\medskip

Next we describe the map $\Psi^g$ of \eqref{e:Psig} in terms of this parametrization,
i.e.~as a map which assigns to each multisegment an (equivalence class of an) involution
with signed fixed points.  We shall do so through a detailed example.
 Suppose $n=11$ and $\lam= (4, 4 ,3,3,3,3,2,2,1,1,0)$ and $\tau$ is the 
 multisegment $\left \{\{0,1,2,3,4\}, \{1,2,3\}, \{2\}, \{3\}, \{3\},\{4\}\right \}$.
 Since there are five even entries of $\lam$ and six odd ones, 
 we will assign to $\tau$ an involution in $S_{11}$ with
 signed fixed points of signature $(5,6)$.
 We first start with a diagram where the entries of $\lam$ are arranged in
 columns and replaces by signs according to their parity.
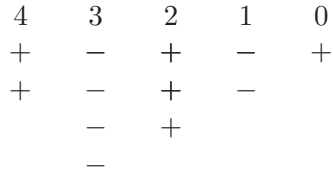
\[\xy
(0,0)*{4};
(10,0)*{3};
(20,0)*{2};
(30,0)*{1};
(40,0)*{0};
 (0,-5)*{+};  (10,-5)*{-};  
 (10,-5)*{-};  (20,-5)*{+}; 
 (20,-5)*{+};  (30,-5)*{-}; 
 (30,-5)*{-};  (40,-5)*{+}; 
(10,-10)*{-}; (20,-10)*{+}; 
(20,-10)*{+}; (30,-10)*{-}; 
(0,-10)*{+}; (10,-15)*{-}; (10,-20)*{-};
(20,-15)*{+};
\endxy
\]
Start with the longest connected component of $\tau$.  If starts at 0 and ends at 4.
We connect a $+$ in the 0 column with a $+$ in the $4$ column.  Since these
columns have the same parity, we invert one sign in each of the intermediate columns.
(If the signs we connected had opposite parities, we would need no such inverting.)
The picture we get is
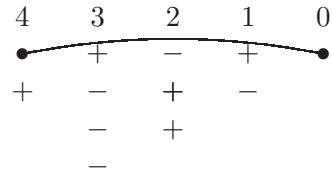
\[\xy
(0,0)*{4};
(10,0)*{3};
(20,0)*{2};
(30,0)*{1};
(40,0)*{0};
 (0,-5)*{\bullet};  (10,-5)*{+};  
(20,-5)*{-}; 
(30,-5)*{+}; 
(40,-5)*{\bullet}; 
(10,-10)*{-}; (20,-10)*{+}; 
(20,-10)*{+}; (30,-10)*{-}; 
(0,-10)*{+}; (10,-15)*{-}; (10,-20)*{-};
(20,-15)*{+};
(0,-5)*{}="A";
(40,-5)*{}="B";
"A"; "B" **\crv{(20,-1)}; 
\endxy
\]
The next longest connected component in $\tau$ connects 1 to 3.  So we take a $-$ in the 1 column
and connect it to a $-$ in the 3 column.  (We never use signs which were changed in previous steps.)
Since we connected two signs of the same parity, we change a $+$ sign to a $-$ sign in the intermediate
column labeled 2.
We obtain:
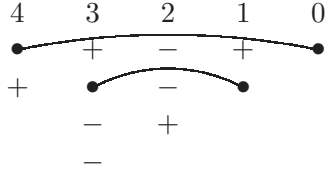
\[
\xy
(0,0)*{4};
(10,0)*{3};
(20,0)*{2};
(30,0)*{1};
(40,0)*{0};
 (0,-5)*{\bullet};  
 (10,-5)*{+};  
(20,-5)*{-};
(30,-5)*{+}; 
(40,-5)*{\bullet}; 
(0,-10)*{+}; 
(10,-10)*{\bullet}; 
(20,-10)*{-}; 
(30,-10)*{\bullet}; 
(10,-15)*{-}; 
(10,-20)*{-};
(20,-15)*{+};
(0,-5)*{}="A";
(40,-5)*{}="B";
"A"; "B" **\crv{(20,-1)}; 
(10,-10)*{}="C";
(30,-10)*{}="D";
"C"; "D" **\crv{(20,-5)}; 
\endxy
\]
Now we come to a final flattening procedure.  In this step, we want
to produce a linear array of $+$'s and $-$'s and connected dots.  To
do so, we throw away the numbers and 
collapse each column of the above diagram to the same
height, but do make any identifications in the process and do not mix
adjacent columns.  This step is
ambiguous.  For instance, we could collapse the above diagram to obtain
either of the following diagrams (among many others).
\medskip

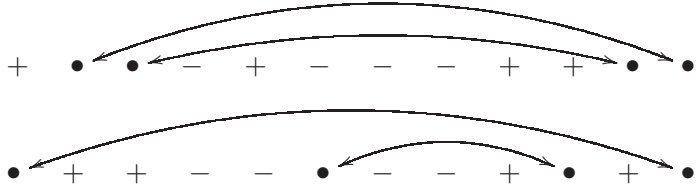
\[
{\xymatrixcolsep{.75pc} \xymatrixrowsep{2pc} \xymatrix{ + & \bullet
    \ar@/^2pc/[rrrrrrrrrr] & \bullet\ar@/^1pc/[rrrrrrrr] & - & + & -&
    -& - &+ &+ &\bullet\ar@/_1pc/[llllllll] &\bullet\ar@/_2pc/[llllllllll] } }
\]

\bigskip

\[
{\xymatrixcolsep{.75pc} \xymatrixrowsep{2pc} \xymatrix{ \bullet     \ar@/^2pc/[rrrrrrrrrrr] 
& + & + & - & -
& \bullet\ar@/^1pc/[rrrr] &
    -& - &+ &\bullet\ar@/_1pc/[llll] &+ &\bullet\ar@/_2pc/[lllllllllll] } }
\]

\medskip

\noindent
Each of these diagrams may be interpreted as an involution with signed fixed points on
11 elements in the obvious way.  Although individually they are not well-defined they
automatically belong to the same equivalence class described above, so indeed
determine a well-defined orbit of $K$ on $G/P$.   Thus we have
taken the multisegment parameter for $\gPH_{11}(\lam)$ and defined a
signed involution parameter for $\gPR_{11}(\lam)$.

The example clearly generalizes to give a map from $\gPH_n(\lam)$ to 
$\gPR_n(\lam)$ in general.  It is not difficult to verify that
this map indeed coincides with $\Psi^g$ of \eqref{e:Psig}.

Next we remark that the explicit details of the map $d_\bbH$ are given
in \cite{Ze}.  More precisely, there is an obvious correspondence between
multisegments $\caM_\circ(\lam)$ and the parameter set
$\PH_n(\lam)$.  
{It takes a multisegment represented by
$\tau =\{ \{a_1,\dots,b_1\},\dots,\{a_r,\dots,b_r\}\}\in\caM(\lam)$
  satisfying $\Re\frac{a_1+b_1}2\geq \Re\frac{a_2+b_2}2\geq\dots\geq
  \Re\frac{a_r+b_r}2$ to the parameter $(\bbH_P,\delta)$ where $P$ corresponds
  to the parabolic subalgebra whose (ordered) Levi factor
  is
  \[
  \frl=\frg\frl(b_1-a_1+1)\oplus\dots\oplus \frg\frl(b_r-a_r+1),\]
 and $\delta = \delta_1 \boxtimes \cdots \boxtimes \delta_r$ where
$\delta_i=\St\otimes\C_{\frac {a_i+b_i}2}.$
We already remarked above that there is a also natural correspondence 
between $\caM_\circ(\lam)$ and $\gPH_n(\lam)$.  With these identifications 
in place,
the map $d_\bbH \; : \PH_n(\lam) \rightarrow \gPH_n(\lam)$ 
is simply the identity map on multisegments. }

Finally, we discuss $d_\R$.  It takes a parameter
$\gamma \in \PR_n(\lam)$ and produces an element of $\gPR_n$,
which we have now identified with the the orbits of 
$K \simeq \GL(p,\C) \times \GL(q,\C)$ on $G/P$.
In turn we may identify such orbits with a subset of $K$ orbits on $G/B$,
namely the ones which are maximal in the preimage under the projection
from $G/B$ to $G/P$ of an orbit on $G/P$.  Using Beilinson-Bernstein localization,
the orbits of $K$ on $G/B$ correspond to irreducible Harish-Chandra
modules for $\U(p,q)$
with trivial infinitesimal character.  Unwinding these identifications,
we can interpret the map $d_\R$ as sending a parameter $\gamma \in \PR_n(\lam)$ to 
an irreducible Harish-Chandra module for $\U(p,q)$, and it is this
correspondence we seek to describe explicitly.  Let $\irr^\mathrm{reg}(\gamma)$
denote an irreducible Harish-Chandra module with regular infinitesimal character
which translates by a ``push to walls'' translation functor to $\irr(\gamma)$.
The paper \cite{v:ic4} assigns to $\irr^\mathrm{reg}(\gamma)$ an irreducible
Harish-Chandra module, say $\irr^\vee(\gamma)$, for $\U(p,q)$.  Then,
with all the identifications in place, the map $d_\R$ takes $\gamma$
to $\irr^\vee(\gamma)$.
Each of these identifications (and the map of \cite{v:ic4}) can be made
very explicit.

We have thus sketched the explicit details of the ingredients $d_\R$,
$d_\bbH$, and $\Psi^g$ in the definition of $\Psi_{n,n}$.  It is thus
possible to compare $\Psi_{n,n}$ to $\Gamma_{n,n}$ directly
and check they coincide.  We omit further details.\qed

\end{document}